\documentclass[10pt,leqno]{amsart}

\usepackage{hyperref, url, doi}
\usepackage{multicol} 
\usepackage{graphics,epsfig,color}%in case of figures
\usepackage{amsmath,amsthm,amssymb}
\usepackage{dsfont}
\usepackage{framed}
\usepackage{amsfonts}
\usepackage{cite}
\usepackage[]{changes}
\usepackage{comment}
\usepackage{tikz}
\usetikzlibrary{positioning,decorations,arrows,patterns,shapes,backgrounds}

\title{Figure's annotation}
\theoremstyle{plain}
\newtheorem{theorem}{Theorem}[section]
\newtheorem{remark}[theorem]{Remark}
\newtheorem{example}[theorem]{Example}
\theoremstyle{plain}

\newtheorem{lemma}[theorem]{Lemma}
\newtheorem{proposition}[theorem]{Proposition}

\newtheorem{assumption}[theorem]{Assumption}

\numberwithin{equation}{section}

\providecommand{\abs}[1]{\lvert#1\rvert}%absolute value
\providecommand{\norm}[1]{\lVert#1\rVert}%norm
\newcommand{\CC}{\mathbb{C}}
\newcommand{\RR}{\mathbb{R}}
\newcommand{\NN}{\mathbb{N}}

\newcommand{\Dt}{\dfrac{d}{dt}}

\newcommand{\intO}{\int_{\Omega}}

\DeclareMathOperator{\divv}{div}
\DeclareMathOperator{\tr}{tr}
\DeclareMathOperator*{\esssup}{ess\,sup}
\DeclareMathOperator{\Sym}{Sym}
\newcommand{\Ss}{\mathbf{S}}

\newcommand{\dx}{\mathrm{d}\, x}
\newcommand{\ds}{\mathrm{d}\, s}
\newcommand{\dt}{\mathrm{d}\, t}

\DeclareMathOperator{\D}{D}

\newcommand{\del}{\partial}

\def\XXint#1#2#3{{\setbox0=\hbox{$#1{#2#3}{%
\int}$ }
\vcenter{\hbox{$#2#3$ }}\kern-.6\wd0}}

\def\edc{\end{document}}
\numberwithin{equation}{section}

\title[Compressible non-Newtonian fluids with initial vacuum]{Remark on the local well-posedness of compressible non-Newtonian fluids with initial vacuum}
\author[Al Baba]{Hind Al Baba}
\author[Al Taki]{Bilal Al Taki}
\author[Hussein]{Amru Hussein} 
\address{Department of Mathematics,
	RPTU Kaiserslautern-Landau (formerly TU Kaiserlautern),  Paul-Ehrlich-Stra\ss e 31, 67663 Kaiserslautern, Germany}
\email{hind.albaba@gmail.com}
\email{bilal.altaki.math@gmail.com}
\email{hussein@mathematik.uni-kl.de}

\thanks{The first and the second author have  been supported partially by the Nachwuchsring – Network for the promotion of young scientists – at RPTU Kaiserslautern-Landau. All three authors have been  supported by MathApp – Mathematics Applied to Real-World Problems - part of the Research Initiative of the Federal State of Rhineland-Palatinate, Germany
}

\begin{document}
%\listofchanges
\maketitle

\begin{abstract}
We discuss in this short note the local-in-time strong well-posedness of the compressible Navier-Stokes system for non-Newtonian fluids on the three dimensional torus. We show that the result established recently by Kalousek, M\'{a}cha, and Ne\v{c}asova in \doi{10.1007/s00208-021-02301-8} can be extended to the case where vanishing density is allowed initially. Our proof builds on the framework developed by Cho, Choe, and Kim in \doi{10.1016/j.matpur.2003.11.004} for compressible Navier-Stokes equations in the case of Newtonian fluids.
 To adapt their method, special attention is given to the elliptic regularity 
of a challenging nonlinear elliptic system. We show particular results in this direction, however, the main result of this paper is proven in the general case 
when  elliptic $W^{2,p}$-regularity  is imposed as an assumption.
 Also, we give a finite time blow-up criterion.
\end{abstract}

\section{Introduction}
The aim of this paper is to show the existence, uniqueness, and continuous dependence on the data of local-in-time strong solutions 
to the Navier–Stokes equations describing non-Newtonian compressible fluids. Here, the main challenge arises from the fact that we admit  initial densities $\rho_0\geq 0$ vanishing on some subset, that is, there is some vacuum initially. 
For simplicity we restrict ourself to the $d$-dimensional torus, i.e., $\Omega=\mathbb{T}^d$. We consider for $T\in (0,\infty]$
\begin{equation}\label{depart}
\begin{array}{cl}
\del_{t}\rho+\divv(\rho u)=0 & \hbox{ in } \Omega \times (0,T),\\
\del_{t}(\rho u)+\divv(\rho u\otimes u)-\divv \Ss u +\nabla p=\rho f & \hbox{ in }  \Omega \times (0,T),\\
\rho(0)=\rho_0 \quad \hbox{and}
\quad u(0)=u_0 &\hbox{ in }  \Omega,
\end{array}
\end{equation}
where $u\colon \Omega\times (0,T)\rightarrow \RR^d$ is the velocity field of the fluid, $\rho\colon \Omega \times (0,T) \rightarrow \RR$ is its density, the pressure $p=p(\rho)$ is a function of the density $\rho$, where $p\colon [0, \infty) \rightarrow \RR^+$ is assumed here to be a $C^2$-function of the density.
Here, $\Ss u=(\Ss_{ij}u)_{1\leq i,j\leq d}$ represents the stress tensor and $f$ is the external body force.
We restrict ourselves to the following constitutive law  
\begin{align}\label{stress-tensor}
\begin{split}
\Ss u&= 2\mu(|\D(u)|^2)\D(u)+\lambda(\divv u)\divv u \, \mathbb{I},
\end{split}
\end{align}
where $\D(u)=\tfrac{1}{2}(\nabla u + (\nabla u)^T)$ is the symmetric part of the gradient, $\mathbb{I}$ denotes the identity in $\RR^d$, and
$\mu\in C^1([0,\infty),\RR)$ and $\lambda\in C^1(\RR,\RR)$ 
satisfying some ellipticity conditions~\eqref{eq:ellipticity2}--\eqref{eq:ellipticity} discussed below.

The result presented here fills the gap between two types of results:  
On one side, in \cite{sarka, sarka2}, Kalousek, M\'{a}cha,  and Ne\v{c}asova proved the local-in-time existence and uniqueness of a strong solution %in maximal $L^q_t$-$L_t^p$-regularity spaces 
to system \eqref{depart}-\eqref{stress-tensor}  working in Lagrangian coordinates, where it is crucial for their method to exclude vacuum at initial state, that is to assume $\rho_0\geq \delta$ for some constant $\delta>0$, see also   \cite{Pruss2007} for results in the same direction.
On the other side, considering a compressible Newtonian fluid, that is system \eqref{depart}-\eqref{stress-tensor} with $\mu,\lambda$ constant, Cho, Choe, and Kim 
proved local strong existence and uniqueness results in \cite{choe_kim1} including vacuum at initial state, that is $\rho_0\geq 0$ assuming certain compatibility conditions. We aim here to extend these results to include both non-Newtonian fluids \emph{and} initial vacuum adapting the ideas developed in \cite{choe_kim1}. Also, the finite time blow-up criterion from \cite{choe_kim1} carries over to the situation discussed here.

Depending on the characteristics of a fluid or gas one  distinguishes different types of Navier-Stokes equations. Its classical version deals with  incompressible Newtonian fluids with constant density and constant viscosity, however there are more complex  fluids described by the negations of these properties, that is, compressible, non-Newtonian fluids with non-constant viscosities and there are combinations of these different characteristics, and  also different descriptions of the pressure in the compressible case.
In terms of mathematics, considering the different space dimensions  and  weak or strong solutions also gives rise to different situations  requiring very different methods. In most of these situations vanishing initial densities have been studied already. 
The starting point for the study of initial vacuum seems to be the work by 
Salvi and Stra\v{s}kraba \cite{Salvi1993} for strong and Lions \cite{Lions_1993a, Lions_1993b, Lions_1996book, Lions_1998book} for weak solutions in dimensions larger or equal to two. 
The work \cite{Salvi1993} already contains the compatibility conditions \eqref{comp-cond} given below of which Cho, Choe, and Kim showed that it is indeed necessary, compare \cite[Theorem 9]{choe_kim1}. The problem of vanishing initial vacuum attracted a lot of interest and there is a large literature on the subject, see e.g. for recent results \cite{Li2023, Danchin2019, Huang2021, Li2020, Li2019}.
  Strong solutions in the case of non-Newtonian fluids with vanishing initial density seem to be addressed so far only in the one dimensional situation, compare e.g. \cite{Muhammad2020, Fang2018, Yuan2012, Yuan2016, Song2013, Nishida1986} and the references therein, and also \cite{Straskraba2010} for an overview on one-dimensional models in fluid mechanics.
  In the three dimensional case there are 
 results for non-Newtonian fluids dealing with weak solutions, see \cite{Zhikhov, Mamotov} and with dissipative solutions, see \cite{Fei-Nov-diss} and also \cite{Feireisl2021} for a weak-strong uniqueness result. Here, we complement this by discussing strong solutions for non-Newtonian fluids in dimension three.
  
  %\footnote{\BT{Not really. There are some results also in 3d but with some conditions in the viscosity coefficients or the growth of the pressure but we do not have to comment on all existence result. We should also specify on what kind of sol we're talking about (weak, strong, dissipative). it is really a broad topic \cite{Zhikhov} , several of Feireisl and co-authors about dissipative solution...}}.   

The primary challenge in exploring strong solutions %for compressible non-Newtonian equations 
to \eqref{depart} lies in establishing higher norm estimates on the velocity. In \cite{sarka, sarka2}, the authors successfully tackle this challenge by assuming an absence of vacuum in the initial state. In the linearisation of \eqref{depart}, the function $u_t$ in \eqref{depart} %(see also \eqref{iterating-scheme}) 
is multiplied by the non-vanishing  positive function $\rho$, cf. \cite[Equation (19)]{sarka},  and consequently higher estimates on the velocity are deduced using maximal $L_t^p$-$L^q_x$-regularity methods such as the Weis multiplier theorem. 
This approach becomes ineffective in the presence of initial vacuum. Therefore, the  strategy here is to  adapt the method introduced in \cite{choe_kim1} for the Newtonian  case to the non-Newtonian  case studied here. To this end, it is crucial to pay  attention to the elliptic system associated with equation \eqref{depart}. 
Formally, system~\eqref{depart} consists  of a coupled hyperbolic transport-type equation for the density and a  parabolic diffusion-type equation for the velocity.   
However, for vanishing density the velocity equation becomes a mixed elliptic-parabolic problem, see e.g. \cite{pluschke} for a related setting. To deal with the highly non-linear elliptic part some additional estimates are needed.  Thus, for Newtonian fluids, the linear elliptic regularity plays an important role, e.g. in \cite[Section 5]{choe_kim1}. However, 
$W^{2,p}$-elliptic regularity for non-linear elliptic systems does not hold in general and therefore, we impose it
here by  Assumption~\ref{assumption-elliptic-reg}.  This assumption and the regularity for non-linear elliptic systems is discussed in detail in Section~\ref{section-elliptic-problem}. With this at hand we can derive the \textit{a priori} bounds needed to prove our main result on  the strong local well-posedness  presented in  Section~\ref{Section-Main-result}. The proof is given in Section~\ref{sec:proof}.

\subsection*{Notation}
We consider throughout the paper the case $\Omega=\mathbb{T}^d$, where $\mathbb{T}^d$ denotes the $d$-dimensional flat torus. By $L^p(\Omega)$ for $p\in [1,\infty]$ we denote the usual Lebesgue spaces setting
$L^p_0(\Omega):=\{u\in L^p(\Omega)\colon \int_{\Omega}u=0\}$. 
The Sobolev spaces with periodic boundary conditions are denoted by
 $W^{m,p}(\Omega)$ where $m\in \mathbb{N}$ and $p\in [1,\infty]$. We use the notation $H^m(\Omega):=W^{m,2}(\Omega)$ and set $H^{-m}(\Omega)$ to be the dual space of $H^m(\Omega)$. 
  Recall that for $u\in W^{1,p}(\Omega)\cap L^p_0(\Omega)$ and $p\in [1,\infty)$ a Poincar\'e inequality $\norm{u}_{L^p(\Omega)}\leq C \norm{\nabla u}_{L^p(\Omega)}$ holds. For a Banach space $X$ and an interval $I\subset \RR$ we denote by $L^q(I,X)$ for $q\in [1,\infty]$ the usual Bochner spaces and by $C(I,X)$ the space of continuous functions on $I$ with values in $X$. If $H$ is a Hilbert space, then we denote its scalar product by $\langle\cdot,\cdot\rangle_H$, and we omit the subscript if there is no ambiguity. For two matrices $A,B$ of the same size we set $A:B=\sum_{i,j}a_{ij}b_{ij}$. Throughout the paper we consider only spaces over $\RR$.

\section{Main result}\label{Section-Main-result}

Considering the stress tensor in \eqref{stress-tensor} we  impose the following ellipticity conditions. Let
\begin{align}\label{eq:ellipticity2}
    \varepsilon_{\mu}>0, \quad \hbox{and}\quad \varepsilon_{\lambda}\in \RR \quad \hbox{with} \quad
    2\varepsilon_{\mu} + 3\varepsilon_{\lambda} >0
\end{align}
be constants such that the continuously differentiable functions $\mu\colon [0,\infty)\rightarrow \RR$ and $\lambda\colon  \RR\rightarrow \RR$ satisfy
\begin{align}
\begin{split} \label{eq:ellipticity}
\mu(s)\geq&\varepsilon_{\mu},\quad  \mu(s) \,+\, 2\,s\mu^{\prime}(s) \geq \varepsilon_{\mu} \quad \hbox {for all } s\geq 0, \quad \mathrm{and}, \\ % \quad %\lambda(r)\geq\varepsilon_{\lambda},
 \lambda(r)\geq&\varepsilon_\lambda, \quad  \lambda(r)\,+\, r\lambda^{\prime}(r)\geq \varepsilon_{\lambda} \quad \hbox {for all } r\in \RR.
 \end{split}
\end{align}
These requirements become more transparent when, after applying the chain rule, one rewrites 
\begin{multline}\label{eq:S_chainrule}
\divv \Ss u = 2\mu(|\D (u)|^2)\divv \D  u + 4\mu'(|\D (u)|^2)\big(\sum_{i=1}^d \langle \partial_i \D (u),\D (u)\rangle(\D (u))_i\big) \\ +\left(\lambda(\divv u) +\lambda'(\divv u)\divv u \right)\nabla \divv u.
\end{multline}
Also, \eqref{eq:ellipticity2}--\eqref{eq:ellipticity} imply together with the regularity $\mu,\lambda\in C^1$ that $\Ss$ is a strictly monotone operator on $W^{1,p}(\Omega)$.

Next, we present 
our assumption on the elliptic regularity of the nonlinear elliptic system associated with~\eqref{depart}. This assumption is discussed in details in the subsequent Section~\ref{section-elliptic-problem}. 

%\added{To guarantee the uniqueness of solution here, we have to add that $u$ should have also zero variation mean}

\begin{assumption}[$W^{2,p}$-regularity]\label{assumption-elliptic-reg}
 Let $p\in (1,\infty)$, then we assume that the non-linear elliptic problem 
\begin{equation}\label{eq:ellip-prob} 
-\divv \Ss u = f,
\end{equation}
has for each $f\in L_0^p(\Omega)^d$ a unique solution $u\in W^{2,p}(\Omega)^d$ with $\int_{\Omega} u\, dx = 0$, and  there exists a constant $C>0$ such that
\begin{align}\label{eq:W2_estimate}
\norm{u}_{W^{2,p}(\Omega)}\leq C \norm{f}_{L^p(\Omega)} \quad \hbox{for all } f\in L_0^p(\Omega)^d.
\end{align}
\end{assumption}

 %\added{add also that the $ \int_{\Omega} u_0\, dx = 0$}

\begin{theorem}[Main result]\label{Main-Theorem} 
Let $d=3$,
\begin{align*}
q\in (3,\infty) \quad \hbox{and} \quad q_{0}=\min\{6,q\}.
\end{align*}
Suppose that $\mu\in C^1([0,\infty),\RR)$  and $\lambda\in C^1(\RR,\RR)$ are functions satisfying \eqref{eq:ellipticity2}--\eqref{eq:ellipticity}   such that Assumption~\ref{assumption-elliptic-reg} holds for $p=2$ and $p=q_0$, and
let $p(\cdot)\in C^{1}([0, \infty),\RR^+)$. Assume that the data %$(\rho_0, u_0, f)$ 
satisfy the following regularity conditions 
\begin{eqnarray}
\rho_{0}\in  W^{1,q}(\Omega), & & u_{0}\in  H^{2}(\Omega) \quad\hbox{with}\quad \int_{\Omega} u_0\, dx = 0, \label{regcond1}\\
f\in C([0,T]; L^{2})\cap L^{2}(0,T; L^{q}(\Omega)) & \mathrm{and} & f_{t}\in L^{2}(0,T; H^{-1}(\Omega)) \label{regcond2},
\end{eqnarray} 
 and the compatibility condition %(besser $q\in (3,6]???$)
\begin{equation}\label{comp-cond}
-\divv \Ss u_0 + \nabla p(\rho_0) = \rho_0^{1/2} g \quad \mbox{for some }\quad g\in L^2(\Omega).
\end{equation}
Then there exist a time $T_{\ast}\in (0,T]$ and a unique strong solution $(\rho, u)$ to the non-linear problem~\eqref{depart} such that 
\begin{eqnarray}\label{reg-sol}
\begin{split}
\rho\in C([0,T_{\ast}];  W^{1,q_{0}}(\Omega)), & & u \in C([0,T_{\ast}];  H^{2}(\Omega))\cap L^{2}(0,T_{\ast}; W^{2,q_{0}}(\Omega)), \\
\rho_{t}\in C([0,T_{\ast}]; L^{q_{0}}(\Omega)), & & u_{t}\in L^{2}(0,T_{\ast}; H^{1}(\Omega)) \quad \mathrm{and} \quad \sqrt{\rho} u_{t}\in L^{\infty}(0,T_{\ast}; L^{2}(\Omega)).
\end{split}
\end{eqnarray}
\noindent
Furthermore, we have the following blow-up criterion: If $T^{\ast}$ is the maximal existence time of the strong solution $(\rho, u)$ and $T^\ast<T$, then 
\begin{equation}\label{blowup}
\underset{t \rightarrow T^\ast}\limsup \big( \| \rho \|_{W^{1,q_0}} + \| u(t) \|_{H^{1}}\big) = \infty.
\end{equation}
\end{theorem}
\begin{remark}[Notion of strong solutions]
The notion of strong solutions used here aligns with that in \cite{choe_kim1}. Specifically, a strong solution to \eqref{depart} is a weak solution that satisfies \eqref{depart} almost everywhere in $(0, T^*) \times \Omega$ and adheres to the properties in \eqref{reg-sol}.
\end{remark}

The regularity and compatibility conditions on the data assumed here for non-Newtonian fluids as well as the regularity class of the solution agree with those in the Newtonian case discussed in \cite[Theorem 7]{choe_kim1}. The statement is valid also for $d=2$, where however improved regularity can be expected. For $d=1$, the problem is already well-understood. %, cf. e.g. \cite{Muhammad2020} and the references therein  \footnote{\BT{This paper is for weak solution, not strong no?}}. 
In particular for $d=1$, the $W^{2,p}$-estimates in Assumption~\ref{assumption-elliptic-reg} hold in general as discussed in Proposition~\ref{1d-estimate} below.
For higher dimensions, as discussed in the introduction,  we have to assume 
$W^{2,p}$-elliptic regularity by Assumption~\ref{assumption-elliptic-reg}, because for vanishing density the velocity equation becomes a mixed elliptic-parabolic problem, and we need this assumption here to deal with the highly non-linear elliptic part.
Compared to the result in \cite{sarka, sarka2}, where it is assumed that $\rho_0\geq\delta> 0$, we obtain slightly less spatial regularity for the density compared to their result which has $\rho\in W^{1,p}(0,T_\ast;W^{1,q}(\Omega))$. The maximal $L^p_t$-$L^q_x$-regularity for $u$ obtained in \cite{sarka} for $f=0$,   $q\in (3,\infty)$, and $p\in (\tfrac{2q}{q-3},\infty)$, which is included in its correction \cite{sarka2}, is different compared  to the maximal  $L^2_t$-$L^{q_0}_x$-regularity obtained here where we admit however $f\neq 0$. Note that the regularity on $\mu,\lambda$ and $p$ can be weakened slightly  here compared to \cite{sarka} by using uniform energy bounds.

\section{Second order $L^{p}$-estimates for the non-linear elliptic system}\label{section-elliptic-problem}
In this section, we discuss Assumption~\ref{assumption-elliptic-reg}.
For $\mu,\lambda$ constants \eqref{eq:ellip-prob} becomes the linear Lam\'e system and provided the ellipticity conditions \eqref{eq:ellipticity2}--\eqref{eq:ellipticity} hold, then $W^{2,p}$-estimates of the form \eqref{eq:W2_estimate} for its solution  follow, cf. e.g. \cite[Section 5]{choe_kim1} or the general elliptic theory in \cite[Chapter 6]{pruss_simonett} and more particularly \cite{Mitrea_Lame, Pruss2007}.

One prototype problem for non-constant $\mu$ (with $\lambda \equiv 0$)
 is the ($p$-$\delta$)-structure generalizing $p$-Laplacians, where $\mu(|\D (u)|^2)=(\delta +|\D (u)|^{p-2})$ with $\delta \geq 0$,  $p\in(1,\infty)$.
 Existence and uniqueness of weak solutions in $W^{1,p}(\Omega)^d$ for this and similar problems can be derived by
the theory of monotone operators or by the calculus of variations, cf. e.g. \cite{Zeidler_II, Lions_book} and \cite{Necas_book}, respectively, 
where ellipticity conditions such as \eqref{eq:ellipticity2}--\eqref{eq:ellipticity} are essential.

Studying the higher regularity of solutions to \eqref{eq:ellip-prob} presents significant challenges due to the highly non-linear nature of this problem, which renders the well-known Caldéron-Zygmund theory ineffective. Even when replacing the symmetric part of the gradient with the gradient itself, the problem remains poorly understood. Notably, Uhlenbeck's renowned paper \cite{Uhlenbeck} addresses this issue by demonstrating that the solution to the system $-\divv(\mu(|\nabla u|^2) \nabla u) = 0$ is globally smooth, precisely in $C^{1, \alpha}_{\rm loc}(\Omega, \mathbb{R}^d)$, provided that $\mu$ is a smooth positive function satisfying conditions similar to \eqref{eq:ellipticity}.
We refer interested readers to the recent significant result by Cianchi and Maz'ya, where they established that $\mu(|\nabla u|^2) \nabla u$ belongs to $W^{1,2}(\Omega)$ if and only if $f$ is in $L^2(\Omega)$ under minimal regularity assumptions on the boundary of the domain, see \cite{Cia-Maz} and the references therein. However, it is important to note that a generalization of these results to $W^{2,q}$-estimates on $u$ is currently lacking.

Considering the symmetric part of the gradient adds further complexity to the problems at hand. Regrettably, the techniques employed in \cite{Uhlenbeck} and \cite{Cia-Maz} do not readily lend themselves to the case involving the symmetric part of the gradient. Nonetheless, specific results can be found in  literature. For instance, in \cite{Veiga} Beir{\~a}o da Veiga showed $W^{2,q}$-estimates for the $p$-Laplacian operator with symmetric part of the gradient, however only, when the parameter $p$ is close to $p=2$, see also  \cite{Bers-Ruz} and the references cited therein for more information on this topic. Counterexamples for %higher 
$W^{2,q}$-regularity are discussed for instance in \cite{Schaftingen2009, Escauriaza2017}. %Further results on $W^{2,p}$-regularity for the scalar case are discussed in
%\cite{Caffarelli1989, Krylov2017, Escauriaza1993}, and
%counterexamples for higher regularity are discussed for instance in \cite{Schaftingen2009, Escauriaza2017}.}

In the rest of this section, we present our contribution on the regularity of solution to system \eqref{eq:ellip-prob}, namely, Propositions~\ref{elliptic-estimates_vega}, \ref{prop:elliptic-estimates} and \ref{rem:H2estimate}. 
Here, we focus on the second order estimates. The $W^{1,p}$-estimates and the existence and uniqueness of solutions are discussed in the concrete Example~\ref{ex:ex1}.
To begin with, in Proposition \ref{elliptic-estimates_vega}, we demonstrate for general space dimension $d$ the second order $L^p$-estimates by employing the approach introduced in \cite{Veiga} by Beir{\~a}o da Veiga, where we have to assume additional smallness conditions on the functions $\mu$ and $\lambda$ in addition to the  ellipticity conditions \eqref{eq:ellipticity2}--\eqref{eq:ellipticity}. This includes functions $\mu,\lambda$ close to constants, and $\mu$ with certain ($p$-$\delta$)-structure. 
Here, in addition to \cite{Veiga}, we are able to quantify these smallness assumptions in Example~\ref{ex:ex1}.
In the one-dimensional case $d=1$ assuming only the usual ellipticity conditions \eqref{eq:ellipticity2}--\eqref{eq:ellipticity}, we establish in Proposition \ref{prop:elliptic-estimates} the second order $L^{p}$-\textit{a priori} estimate.  Lastly, in Proposition~\ref{rem:H2estimate}, we show that the $L^2$-second order \textit{a priori} estimate 
follows for any space dimension from the ellipticity conditions \eqref{eq:ellipticity2}--\eqref{eq:ellipticity}.

 %Following the idea of Beir{\~a}o da Veiga in \cite{Veiga} we can give also in higher dimensions  some cases where the $W^{2,p}-$estimates hold. This includes functions $\mu,\lambda$ close to constants, and $\mu$ with certain $p$-$\delta$ structure. 

\begin{proposition}[Second order $L^p$-estimate under smallness conditions]\label{elliptic-estimates_vega}
	Assume that
	\begin{align*}
	\mu\in C^1([0,\infty),\RR),\quad \lambda\in C^1(\RR,\RR), \quad \overline{\lambda}>-\tfrac{1}{2},\quad \hbox{and}\quad p\in (1,\infty).
	\end{align*}
	Let $C_{p,\overline{\lambda}}>0$ be a constants such that the solutions $u$ of the linear Lam\'{e} system
	\begin{align*}
	\divv \D(u)+\overline{\lambda}\nabla\divv u=f  
	\end{align*}
satisfy for all  $f\in L_0^p(\Omega)^d$
	\begin{align*}
	\norm{\nabla \D(u)}_{L^{p}}  +\norm{\nabla \divv u}_{L^p} \leq C_{p,\overline{\lambda}} \norm{f}_{L^p}.
	\end{align*}
	\begin{enumerate}
		\item 
		If  there exists $\alpha \in [0,1)$ and $c>0$ such that for all $s>0$
		\begin{align*}
		\tfrac{1}{\mu(s)} \leq c s^{\alpha/2}, \quad
		C_{p,\overline{\lambda}}\sup_{S \in \Sym(d,\RR)}\left|\overline{\lambda}-\tfrac{\lambda(\tr S)+\lambda'(\tr S)\tr S}{2\mu(|S|^2)}\right|\leq 1,
		\quad \hbox{and} \quad
		\delta:=C_{p,\overline{\lambda}}\sup_{s\geq 0}\tfrac{|2\mu'(s)s|}{\mu(s)} <1,
		\end{align*}
		%If
		%$
		%\tfrac{|\lambda(s)|}{\mu(s)}+C_p\cdot c_{\mu}<1$ for all $s>0$ 
		where $\Sym(d,\RR)$ denotes the space of symmetric real $d\times d$-matrices, then a weak solution $u$ of \eqref{eq:ellip-prob} satisfies %even $u\in W^{2,p}(\Omega)^d$ 
		for  $p>d$ if $\alpha>0$ and for any $p\in(1,\infty)$ if $\alpha=0$, and for $C>0$ depending on $c,\alpha, p, \delta$
		\begin{align*}
		\norm{\nabla \D (u)}_{L^p} \leq 
		C \norm{f}_{L^p}^{1/(1-\alpha)}.
		\end{align*}
		\item 
		For the constant 	$C_{p,\overline{\lambda}}$ we have 
		\begin{align}\label{eq:cp_estimate}
		C_{p,\overline{\lambda}} \leq 
		\begin{cases}
		d^2(p-1)(1+ \tfrac{1+2\overline{\lambda}}{2+2\overline{\lambda}}d(p-1))
		+
		d(p-1)(1+ \tfrac{1+2\overline{\lambda}}{2+2\overline{\lambda}}), & \hbox{if } p\in (2,\infty), \\
		(d^2+1)(1+ \tfrac{1+2\overline{\lambda}}{2+2\overline{\lambda}}), &  \hbox{if } p=2.
		\end{cases}
		\end{align}
	%	\item Existence of weak solutions: ellipticty and growthg condition...
	\end{enumerate}

\end{proposition}

\begin{example}\label{ex:ex1}
	In \cite{Veiga} Beir{\~a}o da Veiga discusses for $\delta\geq 0$ and for some $\gamma\in (1,2]$ examples with a ($\delta$-$\gamma$)-structure of the type
	\begin{align*}
	\mu_{\delta,\gamma}(s)= (\delta + s)^{(\gamma-2)/2} \quad \hbox{and}\quad\lambda\equiv 0, \quad \hbox{i.e.,}\quad 	\mu_{\delta,\gamma}(|\D(u)|^2)= (\delta + |\D(u)|^2)^{(\gamma-2)/2}.
	\end{align*}
	The existence of a unique weak solution  $u\in W^{1,\gamma}(\Omega)$ follows from the classical theory of monotone operators, cf. e.g. \cite{Zeidler_II}, as discussed in \cite[Section 2]{Veiga}.
To find admissible $\gamma$ to apply Proposition~\ref{prop:elliptic-estimates} one estimates with $\overline {\lambda}=0$
	\begin{align*}
	\sup_{s\geq 0}\frac{|2\mu_{\delta,\gamma}'(s)s|}{\mu_{\delta,\gamma}(s)}\leq|\gamma-2|
<
 \begin{cases}
 \frac{1}{d(p-1)(d+(d^2/2)(p-1)+3/2)}, & p\in (2,\infty), \\
 \frac{1}{(3/2)(d^2+1)} & p=2.
 \end{cases}
	\end{align*}
In particular for $d=3$ and the relevant case $p=6$ in Theorem ~\ref{Main-Theorem}, we can quantify Beir{\~a}o da Veiga's assumption for the case of the torus as
\begin{align*}
	|\gamma-2|< \varepsilon:=1/114.
\end{align*}
This carries over also to $\gamma>2$, that is $\gamma\in (2-\varepsilon,2+\varepsilon)$.
The ellipticity conditions~\eqref{eq:ellipticity2}--\eqref{eq:ellipticity} however require $\gamma\geq 2$ and $\delta>0$. The case $p=2$ follows from Proposition~\ref{rem:H2estimate} below. Hence, Assumption~\ref{assumption-elliptic-reg} holds for $p\in \{2,6\}$  if $\mu_{\delta,\gamma}$ is as above with $\delta>0$ and $\gamma\in [2,2+\varepsilon)$.
\end{example}

\begin{proof}[Proof of Proposition~\ref{elliptic-estimates_vega}]
	Let $u$ be a solution of $-\divv \Ss u =f$, then by \eqref{eq:S_chainrule}
	\begin{align*}
	\divv \D (u)+\overline{\lambda}\nabla\divv u = -\tfrac{f}{2\mu} -\left( \tfrac{\lambda+\lambda'\divv u}{2\mu}-\overline{\lambda}\right)\nabla\divv u
	-\tfrac{2\mu'}{\mu}\big(\sum_{i=1}^d \langle \partial_i \D (u),\D (u)\rangle(\D (u))_i\big),
	\end{align*}
	where we wrote for brevity $\mu=\mu(|\D(u)|^2)$, $\mu'=\mu'(|\D(u)|^2)$,   $\lambda=\lambda(\divv u)$, and $\lambda'=\lambda'(\divv u)$. 
Hence, by the requirements on $C_{p,\overline{\lambda}}$ and the assumptions, one estimates
\begin{align*}
\norm{\nabla \D (u)}_{L^p}+\norm{\nabla \divv u}_{L^p}\leq C_{p,\overline{\lambda}}
\norm{f/(2\mu)}_{L^p}  + \norm{\nabla \divv u}_{L^p}+
\delta\norm{\nabla \D (u)}_{L^p}.
\end{align*}
%and  hence
%$(1-\delta)\norm{\nabla \D (u)}_{L^p}\leq C_{p,\overline{\lambda}}
%\norm{f}_{L^p} \norm{\tfrac{1}{\mu}}_{L^\infty}$. 
Now, for $p>d$ using Sobolev embeddings, Poincar\'e's inequality for  $\int_{\Omega}\D (u)=0$, and Young's inequality	
	\begin{multline*}
	\norm{\tfrac{f}{\mu}}_{L^p}\leq 
	\norm{\tfrac{1}{\mu(|Du|^2)}}_{L^\infty}\norm{f}_{L^p} \leq C\norm{\D (u)}_{L^\infty}^{\alpha}\norm{f}_{L^p} 
	\leq C\norm{\nabla\D (u)}_{L^p}^{\alpha}\norm{f}_{L^p} 
	\leq
	 (1-\delta)\norm{\nabla\D (u)}_{L^{p}}
	+  2C\norm{f}^{1/(1-\alpha)}_{L^p},
	\end{multline*}
	where $C$ depends on $p,\alpha,\delta$. For $\alpha=0$ one has directly $\norm{\mu(|\D (u)|^2)^{-1}}_{L^\infty}\leq c$.

%Ellipticity estimates as in Proposition~\ref{elliptic-estimates_vega} 
Estimates on $C_{p,\overline{\lambda}}$ can be related to estimates on the norm of the Riesz transform
$R=\nabla (-\Delta)^{-1/2}$ 
with components $R_j=\partial_{x_j}(-\Delta)^{-1/2}$ on $L_0^p(\Omega)^d$ for $p\in (1,\infty)$. 
Computing the norm of the Riesz transform is a challenging and open problem for $p\neq 2$. Fortunately, there is a number of estimates 
on compact Lie groups such as the torus,  see e.g. \cite{Baudoin, Arcozzi1, Arcozzi2} and the discussions therein.

	Let $u$ be a solution to $-\divv \D(u) -\overline{\lambda}\nabla \divv u=f$ for $f\in L_0^p(\Omega)^d$. Then,
	\begin{align*}
	 -\divv \D(u) -\overline{\lambda}\nabla \divv u = -\tfrac{1}{2}\Delta u -\tfrac{1+2\overline{\lambda}}{2} \nabla \divv u 
	 = -\tfrac{1}{2}\Delta(\mathbb{I} + (1+2\overline{\lambda})(R_{i}R_j)_{ij}) u.	\end{align*}  
	Here $(R_iR_j)_{ij}$ denotes the denotes the block operator matrix in $L_0^p(\Omega)^d$ with entries $R_iR_j$ for $1\leq i,j\leq d$. In vector notation one has $(R_iR_j)_{ij} = \nabla\nabla^T \Delta^{-1}$.
	In particular, using that $\nabla^T\nabla \Delta^{-1}=\mathbb{I}$ one concludes that $(R_iR_j)_{ij}^2=(R_iR_j)_{ij}$,
	and therefore	$$(\mathbb{I} + (1+2\overline{\lambda})(R_iR_j)_{ij})^{-1}= 	(\mathbb{I} - (R_iR_j)_{ij} + \tfrac{1}{(2+2\overline{\lambda})}(R_iR_j)_{ij}).$$
Hence,
\begin{align*}
 \partial_{x_k}\partial_{x_l}u = 	\partial_{x_k}\partial_{x_l}2(-\Delta)^{-1}(\mathbb{I} - \tfrac{1+2\overline{\lambda}}{2+2\overline{\lambda}}(R_iR_j)_{ij}) f = 2R_kR_l(\mathbb{I} - \tfrac{1+2\overline{\lambda}}{2+2\overline{\lambda}}(R_iR_j)_{ij}) f.
\end{align*}
Moreover, one has the  pointwise estimate $\abs{\nabla \D (u)}\leq \tfrac{1}{2}(\abs{\nabla\nabla u} +  \abs{\nabla(\nabla u)^T})= \abs{\nabla^2 u}$, and therefore
\begin{align}\label{eq:Riesz_Cp}
\begin{split}
\norm{\nabla \D (u)}_{L^p} \leq \norm{\nabla^2 u}_{L^{p}} &\leq C^1_{p,\overline{\lambda}} \norm{f}_{L^p},\quad  \hbox{where}\quad C^1_{p,\overline{\lambda}}\leq 2\sum_{1\leq k,l\leq d}\norm{R_kR_{l}}(1+\tfrac{1+2\overline{\lambda}}{2+2\overline{\lambda}}\norm{(R_{i}R_j)_{ij}}), \\
\norm{\nabla \divv u}_{L^p}   &\leq C^2_{p,\overline{\lambda}} \norm{f}_{L^p},\quad  \hbox{where}\quad C^2_{p,\overline{\lambda}}\leq 2(\norm{(R_iR_{j})_{ij}}+\tfrac{1+2\overline{\lambda}}{2+2\overline{\lambda}}\norm{(R_{i}R_j)_{ij}}).
\end{split}
\end{align}
Next, we employ 
estimates for the second order Riesz transforms $(R_iR_j)_{ij}$, that is  $\norm{(R_iR_j)_{ij}}\leq d(p-1)$ and $\norm{R_kR_l}\leq (p-1)$ for $p\in (2,\infty)$ cf. \cite[Theorem 4.1 (i)]{Baudoin}  and also \cite[Theorem 4]{Arcozzi2} where in the coefficient matrix all entries are equal to $1$ or equal to $\delta_{kl}$, respectively. Hence we obtain \eqref{eq:cp_estimate} for $p\in (2,\infty)$ from \eqref{eq:Riesz_Cp}
 and
 for $p=2$ this simplifies since  
$\norm{(\mathbb{I} - \tfrac{1+2\overline{\lambda}}{2+2\overline{\lambda}}(R_iR_j)_{ij})}=\max\{1,1-\tfrac{1+2\overline{\lambda}}{2+2\overline{\lambda}}\}=1$ and $\norm{R_kR_l}= 1$. 
\end{proof}

In the one-dimensional space, system \eqref{eq:ellip-prob} reduces to 
\begin{equation}\label{eq-ellip-1d}
-\del_x \big(\mu(|\del_x u|^2) \del_x u \big) =f.
\end{equation}
\begin{proposition}[Second order $L^p$-estimate in the 1-dimensional case]\label{prop:elliptic-estimates}
	Let $d=1$, $p\in (1,\infty)$, and $\mu, \varepsilon_\mu$ be as in \eqref{eq:ellipticity2}--\eqref{eq:ellipticity}.
	Then for $f\in L^{p}(\Omega)$ a solution $u$ to \eqref{eq:ellip-prob}  belongs to $W^{2,p}(\Omega)$, and we have
	\begin{equation}\label{1d-estimate}
	\intO |\del_x^2 u|^p \, \dx \leq  \varepsilon_\mu^{1-p} \intO |f|^p \, \dx.
	\end{equation}
\end{proposition}
\begin{proof}
	Multiplying both side of \eqref{eq-ellip-1d} by $-|\del_x^2 u|^{p-2} \del_x^2 u$ and integrate over $\Omega$. The left-hand side term gives us
	\begin{align*}
	\intO \del_x \big(\mu(|\del_x u|^2) \del_x u\big) |\del_x^2 u|^{p-2} \del_x^2 u &= \intO \Big( \mu(|\del_x u|^2) \del_x^2 u + 2 |\del_x u|^2 \mu^\prime(|\del_x u|^2) \del_x^2 u\Big) |\del_x^2 u|^{p-2} \del_x^2 u \\
	&= \intO \big( \mu(|\del_x u|^2) + 2 |\del_x u|^2 \mu^\prime(|\del_x u|^2) \big) |\del_x^2 u|^p \, \dx\geq \varepsilon_\mu\intO |\del_x^2 u|^p \, dx.
	\end{align*}
	We estimate the right-hand-side of \eqref{eq-ellip-1d} using H\"older's and Young's inequalities to obtain 
	\begin{align*}
	\intO f \,|\del_x^2 u|^{p-2}\del_x^2 u\,dx
	&\leq \Big(\intO |f|^p\,dx\Big)^{1/p}\, \Big(\intO|\del_x^2 u|^p\,dx\Big)^{(p-1)/p}
	\leq \dfrac{1}{(\varepsilon_\mu)^{p-1} p} \intO |f|^p\,dx +\dfrac{\varepsilon_\mu(p-1)}{p} \intO |\del_x^2 u|^p\,dx.
	\end{align*}
	Therefore, from the above estimates we deduce \eqref{1d-estimate}.
\end{proof}

\begin{proposition}[Second order $L^2$-estimate  in the general case] \label{rem:H2estimate}
 Let $\mu, \lambda$ and $\varepsilon_\mu, \varepsilon_\lambda$ be as defined in \eqref{eq:ellipticity2}--\eqref{eq:ellipticity}.
 If $u\in H^1(\Omega)$ is a weak solution  to system \eqref{eq:ellip-prob}, then  $u\in H^2(\Omega)$ and the following estimate holds
  \begin{align*}
\varepsilon_\mu    \norm{\nabla \D(u)}_{L^2}^2 + 2\varepsilon_\lambda \norm{\nabla \divv u}_{L^2}^2 \leq \tfrac{1}{\varepsilon_\mu}\norm{f}_{L^2}^2.
  \end{align*} 
\end{proposition}
  \begin{proof}
	Indeed, taking the $L^2$-scalar product of equations \eqref{eq:ellip-prob} with $-\partial_i^2 u$ for $1\leq i\leq d$. After performing two integrations by parts -- where the boundary terms vanish due to periodicity -- and applying the chain rule we obtain
		\begin{multline*}
	\intO \langle \divv\big(
	2\mu(|\D(u)|^2)\D(u)\big), \partial_i^2 u\, \rangle_{\RR^d}\dx 
	= 	\intO \langle \partial_i\big(
	2\mu(|\D(u)|^2)\D(u)\big), \nabla\partial_i u\, \rangle_{\RR^{d^2}}\dx \\
	= 	2\intO \langle 
	\mu(|\D(u)|^2)\D(\partial_i u) + 2\mu'(|\D(u)|^2)\langle \D(\partial_i u)\D(u)\rangle_{\RR^{d^2}}\D(u), \D(\partial_i u) \rangle_{\RR^{d^2}}\dx \\
	=	2\intO 
	\mu(|\D(u)|^2)|\D(\partial_i u)|^2 + 2\mu'(|\D(u)|^2)|\D(u)|^2\cdot |\langle\D(\partial_i u),\tfrac{\D(u)}{|\D(u)|}\rangle_{\RR^{d^2}}|^2 \dx.
	\end{multline*}
Here, we have also used that $\langle \tfrac{1}{2}(A+A^T),B \rangle_{\RR^{d^2}}=\langle \tfrac{1}{2}(A+A^T),\tfrac{1}{2}(B+B^T) \rangle_{\RR^{d^2}}$ for matrices $A,B$ since the symmetric matrices form a subspace of $\RR^{d^2}$  the orthogonal projection onto which is given by the symmetric part, and also that $\partial_i\D=\D\partial_i$.
	Then one can estimate using \eqref{eq:ellipticity2}--\eqref{eq:ellipticity}
	\begin{multline*}
	\mu(|\D(u)|^2)|\D(\partial_i u)|^2 + 2\mu'(|\D(u)|^2)|\D(u)|^2\cdot |\langle\D(\partial_i u),\tfrac{\D(u)}{|\D(u)|}\rangle_{\RR^{d^2}}|^2  \\
\geq	\begin{cases}
	\mu(|\D(u)|^2)|\D(\partial_i u)|^2, & \hbox{if } \mu'(|\D(u)|^2) \geq 0, \\
		\big(\mu(|\D(u)|^2)| + 2\mu'(|\D(u)|^2)|\D(u)|^2\big) |\D(\partial_i u)|^2& \hbox{if } \mu'(|\D(u)|^2) <0,
	\end{cases}
	\geq \varepsilon_\mu |\D(\partial_i u)|^2.
	\end{multline*}

	Similarly, after two integrations by parts and applying the chain rule, one has using that $\langle \mathbb{I}, \nabla \partial_i u\rangle_{\RR^{d^2}}=\divv \partial_i u$
			\begin{align*}
\intO \langle\divv\big(\lambda(\divv u)\divv u \, \mathbb{I}
	\big), \partial_i^2 u\rangle_{\RR^d} \dx &= \intO \langle\partial_i\big(\lambda(\divv u)\divv u \,\, \mathbb{I} 
	\big),  \nabla \partial_i u\rangle_{\CC^{d^2}} \dx \\
&= \intO \langle \lambda(\divv u)\divv \partial_i u \, \mathbb{I} + \lambda'(\divv u)(\divv \partial_i u) \divv u\, \mathbb{I} 
,  \nabla \partial_i u\rangle_{\RR^{d^2}} \dx \\
&=
\intO  \big(\lambda(\divv u) + \lambda'(\divv u)(\divv u)\big) |\divv \partial_i u|^2 \dx\\
&\geq \varepsilon_\lambda \intO |\divv \partial_i u|^2 \dx.
	\end{align*}
Using these inequalities eventually leads by Young's inequality and since $|\Delta u| \leq |\nabla \D(u)|$ to
\begin{align*}
\varepsilon_\mu \norm{\nabla \D(u)}^2_{L^2} + \varepsilon_\lambda \norm{\nabla \divv u}_{L^2}^2  \leq  \langle-\divv \Ss u, -\Delta u\rangle_{\RR^d} = \langle f, -\Delta u\rangle_{\RR^d} \leq \tfrac{1}{2\varepsilon_\mu}\norm{f}^2_{L^2} + \tfrac{\varepsilon_\mu}{2}\norm{\nabla \D(u)}^2_{L^2}.
\end{align*}
Thus, 
$u\in H^2(\Omega)$ since the weak solution is already bounded in $H^1(\Omega)$. In the above estimates, the absence of a boundary for $\Omega=\mathbb{T}^d$ has been crucial.   
\end{proof}

\section{Proof of Theorem~\ref{Main-Theorem}}\label{sec:proof}

The idea of the proof adapts 
the overall strategy developed in \cite{choe_kim1} for compressible Newtonian fluids to the non-Newtonian setting. 
Firstly, we shall construct 
approximate solutions $(\rho^k, u^k)$ and establish uniform estimates on these, where in contrast to \cite{choe_kim1} we consider a non-linear approximation. Secondly, we show that the approximate solutions $(\rho^k, u^k)$ for regularized initial data with $\rho_0^{\delta}=\rho_0+\delta$ converges to a solution to system \ref{depart} 
in the limit $k\to \infty$ and $\delta\to 0$. Finally, we discuss the blowup criterion. 
Some details are skipped in this section due to the similarity of our proof with the one of %used to prove 
\cite[Theorem 7]{choe_kim1}. In \cite{choe_kim1} even unbounded domains are included since the estimates are compatible with cut-offs.

\subsection{Construction  of approximate solutions}\label{subsec:approx_sol}
 To construct approximate solutions, we firstly regularize the initial data.  Let $u_0$ and $\rho_0$ be given as in Theorem~\ref{Main-Theorem}.
 For  $\delta>0$, set $\rho_{0}^{\delta}=\rho_{0}+\delta$, and let $u_{0}^{\delta}$ be the solution to the following non-linear elliptic problem 
\begin{equation}\label{initial-elliptic}   
-\divv \Ss u_{0}^{\delta}
= (\rho_{0}^{\delta})^{1/2}g - \nabla p_{0}^{\delta}, \quad \hbox{where}\quad p_{0}^{\delta}:=p(\rho_{0}^{\delta}).
\end{equation}
Here, due to the compatibility condition~\eqref{comp-cond}  $g\in L^{2}(\Omega)$ and by construction $\rho_0^\delta\in W^{1,q}(\Omega)$. Hence, 
by Assumption \ref{assumption-elliptic-reg} (see also Proposition~\ref{rem:H2estimate}), the solution $u_0^\delta$  of  \eqref{initial-elliptic} is unique with $\int_{\Omega} u_0\, dx = 0$ and belongs to $H^2(\Omega)$ uniformly with respect to $\delta$. Hence, the regularized initial data $\rho_0^\delta$ and $u_0^\delta$ satisfy also the assumptions of  Theorem~\ref{Main-Theorem}.

Secondly, we  construct iteratively approximate solutions to system \eqref{depart}. We start by setting $u^0=0$ and 
for $k\geq 1$, let $\rho^k, u^k$ be the unique smooth solutions to 
the following quasi-linear problem
\begin{align}\label{iterating-scheme}
\begin{split}
\rho_{t}^{k}+ u^{k-1}\cdot\nabla \rho^{k} + \rho^{k}\mathrm{div}\, u^{k-1} &= 0,\\
\rho^{k}u^{k}_{t} + \rho^{k}u^{k-1}\cdot \nabla u^{k} -\divv \Ss u^{k}+ \nabla p^{k}  &= \rho^{k} f, \\
(\rho^k, u^k)\vert_{t=0} &= (\rho_0^\delta, u_0^\delta).
\end{split}
\end{align}
Here, we consider smooth approximations of the data and including smooth approximations of $\lambda,\mu$ satisfying uniformly \eqref{eq:ellipticity}. Then
each $k\geq 1$ this problem admits a unique smooth solution on a maximal existence time $T_k>0$ according to the classical existence theorems, see for instance \cite[Chapter 5]{Nec-Malek-book} and \cite[Chapter 6]{Ladyzenskaja_book}, and also \cite{sarka, Pruss2007}.  For simplicity we omit the $\delta$-dependence in the notation of $\rho^k, u^k$.

\subsection{Uniform estimates of approximate solutions}\label{sub4.1}
Next, we shall establish \textit{a priori} estimates  in higher norms on the approximate solutions constructed above.

\begin{lemma}\label{lemma-estimates}
Let $(\rho^k, u^k)$ be the solution to system \eqref{iterating-scheme} with initial data $\rho_0^\delta, u_0^\delta$ for $\delta>0$ where $\rho_0, u_0$ and $f$ are as in Theorem~\ref{Main-Theorem}. Then there is a $0<T^\ast$ with $T^\ast\leq T_k$ for all $k\geq 1$ such that
\begin{eqnarray}\label{eq3.24}
\begin{split}
\sup_{0\leq t \leq T^\ast}\Big[\Vert\rho^{k}(t) \Vert_{W^{1,q_0}(\Omega)}\,+\, \Vert\rho^{k}_{t}(t) \Vert_{L^{q_0}(\Omega)} \,+\, \Vert u^{k}(t) \Vert_{H^{2}(\Omega)}\,+\, \Vert \sqrt{\rho^{k}(t)}u^{k}_{t}(t) \Vert_{L^{2}(\Omega)}\Big] \,\, \\
+\int_{0}^{T^\ast}\Big(\Vert u^{k}(t) \Vert^{2}_{W^{2,q_0}(\Omega)}\,+\, \Vert u^{k}_{t}(t) \Vert^{2}_{H^{1}(\Omega)}\Big) \dt 
\leq C\exp\big(C \exp(C \mathcal{C}_0)\big)
\end{split}
\end{eqnarray}
for all $k\geq 1$, where  $C$ is a generic numerical constant, which does not depend on $k$ and $\delta$ and
\[
\mathcal{C}_0 = \mathcal{C}(\rho_0,u_0) = \intO \rho_0^{-1} \, \vert \divv \Ss u_0 + \nabla p(\rho_0) \vert^2\, \dx =\norm{g}^2_{L^2}.
\]
\end{lemma}

This lemma is the non-Newtonian analogue of  \cite[Eq.  (3.24)]{choe_kim1}. Since the  proofs have many similarities,
we will skip many details in the proof. We shall pay attention to the viscous stress tensor which constitutes the main difference compared to the problem studied in \cite{choe_kim1}.
\begin{proof}[Proof of Lemma \ref{lemma-estimates}] 
The proof of estimate \eqref{eq3.24},  is based on several steps.
Consider the following auxiliary functions as in \cite[Subsec. 3.1]{choe_kim1} for $K>0$ being a fixed integer  
	\begin{equation*}
	\phi_{K}(t) \,=\, \max_{1\leq k\leq K} \sup_{0\leq s \leq t} \Big(1 \,+\, \Vert \rho^{k}(s)\Vert_{W^{1,q_0}}  \,+\, \Vert u^{k}(s)\Vert_{H^1} \Big).
	\end{equation*}

\medskip

{\it Step 1. Estimate for $\Vert u^{k}(s) \Vert_{H^1}$.} We take the scalar product of \eqref{iterating-scheme}$_2$ by $u_t^k$ and integrate over $\Omega$. Then %We obtain the following estimate
\begin{align}\label{3.8}
\begin{split}
\frac{1}{2}\int_{\Omega}\rho^k &|u_t^k|^{2}\,\dx+\, \int_{\Omega} 2\mu \big(\vert \mathrm{D}\, (u^{k}))\vert^{2}\big)\mathrm{D}\, (u^{k}):\nabla u_t^{k}\,\mathrm{d}\,x \,+\, \int_{\Omega} \lambda \big(\mathrm{div}\, u^{k}\big)\mathrm{div}\, u^{k}\mathrm{div}\, u^{k}_{t}\,\mathrm{d}\,x\\
&=\intO \big(\rho^k f-\rho^k u^{k-1}\cdot\nabla u^k\big)\cdot u_t^k\,\dx-\intO \nabla p^k\cdot u_{t}^{k}\,\dx\\
&=\intO \big(\rho^k f-\rho^k u^{k-1}\cdot\nabla u^k\big)\cdot u_t^k  - p_t^k \, \divv u^k \,\dx + \Dt \intO p^k \divv u^k \, \dx.
\end{split}
\end{align} 
%On the one hand, 
The main difference to \cite[Eq. (3.8)]{choe_kim1} are the second and third terms on the left-hand-side.   
These can be rewritten as follows using integration by parts and the fundamental theorem of calculus 
\begin{align*}
\, \int_{\Omega} & 2\mu \big(\vert \mathrm{D}\, (u^{k})\vert^{2}\big)\mathrm{D}\, (u^{k}): \nabla u_t^{k}\,\dx \,+\, \int_{\Omega} \lambda \big(\mathrm{div}\, u^{k}\big)\mathrm{div}\, u^{k}\mathrm{div}\, u^{k}_{t}\,\dx\\
& =
\int_{\Omega}\mu(|\D(u^{k})|^2)\dfrac{d}{dt}|\D(u^k)|^2\,dx+ \dfrac{1}{2}\intO \lambda(\divv u^{k})\dfrac{d}{dt}|\divv u^k|^2\,\dx\\
&= \, \int_{\Omega}\Big(\dfrac{d}{d t}\int_{0}^{\vert \mathrm{D}\, (u^k)\vert^{2}} \mu(s)\,\mathrm{d}\,s\Big)\,\mathrm{d}\,x + \dfrac{1}{2} \, \int_{\Omega}\Big(\dfrac{d}{d t}\int_{0}^{\vert \mathrm{\divv}\, u^k\vert^{2}} \lambda(s)\,\mathrm{d}\,s\Big)\,\dx \\
&=
\int_{\Omega}\Big(\lim_{h\to 0}\dfrac{1}{h}\int_{\vert \mathrm{D}\, (u^k)\vert^{2}(t)}^{\vert \mathrm{D}\, (u^k)\vert^{2}(t+h)} \mu(s)\,\mathrm{d}\,s\Big)\,\mathrm{d}\,x +  \, \lim_{h\to 0}\dfrac{1}{2h}\int_{\Omega}\Big(\int_{\vert \mathrm{\divv}\, u^k\vert^{2}(t)}^{\vert \mathrm{\divv}\, u^k\vert^{2}(t+h)} \lambda(s)\,\mathrm{d}\,s\Big)\,\dx \\
&\geq \varepsilon_{\mu}  \frac{d}{d t}\int_{\Omega}\vert \mathrm{D}\, (u^k)\vert^{2}\,\mathrm{d}\,x + 
 \dfrac{\varepsilon_{\lambda}}{2} \frac{d}{d t}\int_{\Omega}\vert \mathrm{\divv }\, u^k\vert^{2}\,\dx
 = \dfrac{\varepsilon_{\mu}}{2}  \frac{d}{d t}\int_{\Omega}\vert \nabla u^k\vert^{2}\,\mathrm{d}\,x + 
 \dfrac{\varepsilon_{\mu}+\varepsilon_{\lambda}}{2} \frac{d}{d t}\int_{\Omega}\vert \mathrm{\divv }\, u^k\vert^{2}\,\dx,
\end{align*}
where  the last estimate uses the ellipticity condition~\eqref{eq:ellipticity}, and by dominated convergence the integral over $\Omega$ and the time derivative interchange, and in the last inequality several integrations by parts have been applied. 
 So,  \cite[Eq. (3.8)]{choe_kim1} is obtained with $\mu$ and $\lambda$ in \cite[Eq. (3.8)]{choe_kim1} replaced by $\varepsilon_\mu$ and $\varepsilon_\lambda$ here, respectively.  
Consequently, 
we deduce using \eqref{eq:ellipticity} analogously to \cite[Eq. (3.11)]{choe_kim1} for $K\geq k$ that
\begin{equation}\label{3.11}
\int_0^t \Vert \sqrt{\rho^k} u_t^k \Vert_{L^{2}}^2 \, \ds +  \Vert \nabla u^k(t) \Vert_{L^2}^2 \leq C + \int_0^t M(\phi_K) (1 + \Vert \nabla u^k \Vert_{H^1}) \, \ds,
\end{equation}
where here and in the following $M = M(\cdot)\colon [0, \infty) \rightarrow [0, \infty)$ denotes a certain increasing continuous function  with $M(0)=0$ independent of $\delta$, and $C>0$ denotes some universal constant.

To estimate the higher order term $\|\nabla u^k(s) \|_{H^1}$ in \eqref{3.11}, we shall use elliptic regularity. Indeed, remember that by \eqref{iterating-scheme} $u^k$ is the solution of the following non-linear elliptic system
\begin{align}\label{eq:elliptic_estimate}
-\divv \Ss  u^k=F^k, \quad \hbox{where} \quad  F^k:=\rho^k f -\rho^k u_t^k -\rho^k u^{k-1}\cdot \nabla u^k -\nabla p^k.
\end{align}
 Therefore, due to Assumption~\ref{assumption-elliptic-reg} for $p=2$ -- or under slightly different assumptions by Proposition~\ref{rem:H2estimate} --  we obtain  %by Hölder's inequality, \eqref{3.11}, the embedding $H^1(\Omega)\hookrightarrow L^6(\Omega)$
 %\footnote{Later in the proof, precisely when we shall establish higher estimates on the density, the $H^2-$estimate on the velocity is not sufficient. The $W^{2,p}-$estimate on $u$ is needed, and then Assumption \ref{assumption-elliptic-reg} becomes required.}
\begin{align*}
\|u\|_{H^{2}}&\leq  C(\|\rho^k u_t^k\|_{L^{2}(\Omega)}+\|\rho^k u^{k-1}\cdot\nabla u^k\|_{L^{2}}+\|\rho^k f\|_{L^{2}}+\|\nabla p^k\|_{L^{2}})\\
&\leq C(\|\rho^k \|_{L^{\infty}}^{1/2}\|\sqrt{\rho^k}u_t^k\|_{L^{2}}+\|\rho^k\|_{L^{\infty}}\|u^{k-1}\|_{L^{6}}\|\nabla u^k\|_{L^{3}} 
% & \qquad  
+\,\|\rho^k\|_{L^{\infty}}\|f\|_{L^{2}(\Omega)}+\|\nabla p^k\|_{L^{2}})\\
&\leq 2M(\phi_k) \big( 1 + \Vert \sqrt{\rho^k} u_t^k \Vert_{L^2}\big) + C \Vert \rho^k \Vert_{L^\infty} \Vert \nabla u^{k-1} \Vert_{L^2} \Vert \nabla u^k \Vert_{L^3}\\
&\leq 2M(\phi_K) \big( 1  + \Vert \sqrt{\rho^k} u_t^k \Vert_{L^2}\big) + \dfrac{1}{2} \Vert \nabla u^k \Vert_{H^1},
\end{align*}
%Here  by Poincar\'{e}'s inequality $\Vert \nabla u^k \Vert_{H^1}\leq C \Vert \nabla^2 u^k \Vert_{L^2}$ since $\int_{\Omega} \nabla u^k=0$,
and thus (with slight modifications) as in \cite[Eq. (3.12)]{choe_kim1}
\begin{equation}\label{3.12}
\Vert u^k \Vert_{H^2} \leq M(\phi_K) \big(1+ \Vert \sqrt{\rho^k} u_t^k \Vert_{L^2}\big).
\end{equation}
Substituting this into \eqref{3.11} and using Young's inequality, we conclude that  as in \cite[Eq. (3.13)]{choe_kim1}
\begin{align}\label{3.13}
\int_0^t |\sqrt{\rho^k} u_t^k |_{L^2}^2\,ds  + \| u^k(t)\|_{H^1} \leq C + \int_0^t M(\phi_K(s))\, ds \quad \hbox{for all } 1\leq k \leq K.
\end{align}

\medskip

{\it Step 2. Estimate for $\Vert\sqrt{\rho^{k}}u^{k}_{t} \Vert_{L^2}$.} We start by differentiating the momentum equation in \eqref{iterating-scheme} with respect to $t$ to obtain
\begin{multline}\label{diff-momen-time}
\begin{split}
\rho^k u_{tt}^k + \rho^k u^{k-1} \cdot \nabla  u_t^k - 2\divv (\mu(|\D (u^k)|^2) \D (u_t^k)) - 4\divv(\mu^{\prime}(|\D (u^k)|^2) (\D (u^k) : \D (u_t^k)) \, \D (u^k))\\ - \nabla (\lambda(\divv u^k) \divv u^k_t)
 - \nabla (\lambda (\divv u^k) \divv u_t^k \divv u^k) + \nabla p_t^k 
\\  = \rho^k f_t + \rho_t^k (f- u_t^k - u^{k-1}\cdot \nabla u^k )- \rho^k u_t^{k-1} \cdot \nabla u^k.   
\end{split}  
\end{multline}
We take the scalar product of the above equation \eqref{diff-momen-time} by $u_t^k$ and we integrate over $\Omega$. %The third to sixth terms on the left-hand-side arise from the non-linear stress tensor. 
The third and the forth term can be estimated as follows  using \eqref{eq:ellipticity}
\begin{align*}
-2\intO &\divv (\mu(|\D (u^k)|^2) \D (u_t^k)) \cdot u_t^k \, \dx - 4\intO \divv(\mu^{\prime}(|\D (u^k)|^2) (\D (u^k) : \D (u_t^k)) \, \D (u^k)) \cdot u_t^k \,  \dx \\
&=2\intO \mu( |\D (u^k)|^2)\; \vert \D (u_t^k) \vert^2 \, \dx+ 4\intO  \mu^{\prime}  (|\D (u^k)|^2)\; \vert \D (u^k) : \D (u_t^k) \vert^2  \,\dx\\
&\geq	\begin{cases}
	2\displaystyle{\intO} \mu( |\D (u^k)|^2)\; \vert \D (u_t^k) \vert^2 \, \dx & \hbox{if } \mu'(|\D(u^k)|^2) \geq 0, \vspace*{0.3cm}\\
	2\displaystyle{\intO} \big(\mu( |\D (u^k)|^2) + 2|\D (u^k)|^2 \, \,  \mu^\prime( |\D (u^k)|^2)\big)  \vert \D (u_t^k) \vert^2  \, \dx	& \hbox{if } \mu'(|\D(u^k)|^2) <0,
	\end{cases}\\
	& \geq 2\varepsilon_{\mu}\intO \vert \D (u_t^k) \vert^2 \, \dx 
	=  \varepsilon_{\mu}\intO \vert \nabla u_t^k \vert^2  +\vert \divv u_t^k \vert^2 \dx.
\end{align*}
The fifth and sixth terms give by \eqref{eq:ellipticity} rise to
\[
 \int_{\Omega} (\lambda(\mathrm{div}\,u^{k})+ \divv u^k \lambda^\prime(\divv u^k))\big\vert \mathrm{div}\,u_t^{k}\big\vert^{2}\,\mathrm{d}\, x  
 \geq  \varepsilon_\lambda \int_{\Omega}\big\vert \mathrm{div}\,u_t^{k}\big\vert^{2}\dx.
 %\geq {\color{blue}{\epsilon^\prime}}  \intO |\divv u_t^k|^2 \, \dx.
\]
%
%\intO \Big(\rho^k \big( |f_t| \, + |u^{k-1}|\,  |\nabla u^k| \, + |u^{k-1}|\, |\nabla u_t^k|\big) \, |u_t^k| \, + \, |\rho_t^k| \big(|f| + |u^{k-1}|\, |\nabla u^k| \big) \, |u_t^k| + |p_t^k|^2 \Big) \, \dx
Therefore, using the linearised continuity equation we deduce
\begin{multline*}
\dfrac{1}{2} \Dt \intO \rho^k |u_t^k|^2 \, \dx+ \varepsilon_{\mu} \intO \vert \nabla u_t^k \vert^2 \, \dx + (\varepsilon_{\lambda}+\varepsilon_\mu)  \intO |\divv u_t^k|^2 \, \dx\\
\leq  \intO p_t^k \divv u_t^k \, \dx+ \intO \big( \divv(\rho^k u^{k-1}) (u_t^k + u^{k-1} \cdot\nabla u^k - f ) - \rho^k u_t^{k-1} \cdot \nabla u^k + \rho^k f_t\big) \cdot u_t^k \, \dx.
\end{multline*}
Now, we shall estimate all the terms appearing on the right-hand-side. To this end, we use again the linearised continuity equation in \eqref{iterating-scheme} to write similarly to \cite[Estimate (3.14)]{choe_kim1}
\begin{align*}
&\dfrac{1}{2} \Dt \intO \rho^k |u_t^k|^2 \, \dx+ \varepsilon_{\mu}\intO |\nabla u_t^k|^2\, \dx +(\varepsilon_{\lambda} +\varepsilon_\mu)\intO |\divv u_t^k|^2\, \dx\\
&\quad \leq \intO \Big[ 2 |\rho^k| |u^{k-1}| |u_t^k | |\nabla u_t^k| + |\rho^k| |u^{k-1}| |\nabla u^{k-1}| |\nabla u^k| |u_t^k| \\
&\qquad \qquad+ |\rho^k| |u^{k-1}|^2 |u_t|^k |\nabla^2 u^k| + |\rho^k| |u^{k-1}|^2 |\nabla u^k| |\nabla u_t^k|\\
&\qquad\qquad +|\rho^k| |u_t^{k-1}| |u_t^k| |\nabla u^k| +|\nabla p^k| |u^{k-1}| |\divv u_t^{k}|\\
&\qquad \qquad+ |p^\prime(\rho^k)| |\rho^k| |\divv u^{k-1}| |\divv u_t^k| + |\nabla \rho^k| |u^{k-1}|  |f| |u_t^k| \\
&\qquad \qquad +|\rho^k| |\nabla u^{k-1}| |f| |u_t^k| + |\rho^k| |u^{k-1}| |f| |\nabla u_t^k| + |\rho^k| |u_t^k| |f_t|\Big]\, \dx = \underset{j=1}{\overset{11}{\sum}} I_j.
\end{align*}
Following the same lines as in \cite{choe_kim1}, 
we estimate each integral $I_j$ making an extensive use of Sobolev and H\"older inequalities. We deduce analogously to \cite[Eq. (3.19)]{choe_kim1} -- where we have to replace $\mu$ and $\lambda$ by $\varepsilon_\mu$ and $\varepsilon_\lambda$, respectively -- that
\begin{align}\label{3.19}
\| \sqrt{\rho^k} u^k(t) \|_{L^2} + \int_0^t \| u_t^k \|_{H^1}^2 \, ds \leq C (1+C_0) \exp\Big(\int_0^t M(\phi_K(s)) \, ds\Big) \quad \hbox{for all } 1\leq k\leq K.
\end{align}

\medskip

{\it Step 3. Estimate for $\Vert\rho^{k} \Vert_{W^{1,q_0}(\Omega)}$.} The continuity equation for the density is the same comparing compressible non-Newtonian and Newtonian fluids. Hence, the estimates here can be performed  
analogously to \cite[Estimate (3.23)]{choe_kim1}. 
Here, Assumption~\ref{assumption-elliptic-reg} for $p=q_0$ enters when estimating the left hand side in \eqref{eq:elliptic_estimate} from below as in \cite[Estimates (2.22) and (3.20) ff.]{choe_kim1} while the estimate of the right hand side remains as in \cite{choe_kim1}[Estimates (2.22) and (3.20) ff.] leading to  \cite[Estimate (3.23)]{choe_kim1}.
In particular, we have
\begin{equation}\label{3.23}
\| \rho^k(t) \|_{W^{1,q_0}} \leq C \exp\Big( C(1+C_0) \exp\Big( \int_0^t M(\phi_K) \ds \Big) \Big)
\quad \hbox{for all } 1\leq k\leq K.
\end{equation}
%for all $k, 1\leq k\leq K$. 
Thus, we conclude from \eqref{3.13} and \eqref{3.23} that 
\begin{equation*}
\phi_K(t) \leq C \exp\Big(C(1+\mathcal{C}_0) \exp \Big( C\int_0^t M(\phi_K(t)) \, \ds\Big)  \Big),
\end{equation*}
for some increasing $M(\cdot)$ as above.
%continuous function $M=M(\cdot)\colon [0, \infty) \rightarrow [0, \infty)$. 
Hence if we define $\psi_K(t)= \log (C^{-1} \log (C^{-1} \phi_K(t)))$, then we have
\begin{align*}
\psi_K(t) \leq \log (1+ C_0) + C \int_0^t M \Big( C\exp(C\exp(\psi_K(s))\Big) \ds.
\end{align*}
Thanks to this integral inequality, we deduce by a non-linear Gr\"onwall's inequality, cf. e.g. \cite[Theorem 4]{Dragomir},
that there exists a small time $T^\ast\in (0,T]$ depending only on $C_0$ and $C$ such that $\phi_K(T^\ast)\leq C \exp(C C_0)$. Moreover, the following estimates analogous to \cite[Eq. (3.24)]{choe_kim1} hold true
\begin{align}\label{3.24}
\begin{split}
\sup_{0\leq t \leq T^\ast}\Big[\Vert\rho^{k} \Vert_{W^{1,q}(\Omega)}&\,+\, \Vert\rho^{k}_{t} \Vert_{L^{q_0}(\Omega)} \,+\, \Vert u^{k} \Vert_{H^{1}(\Omega)}\,+\, \Vert \sqrt{\rho^{k}}u^{k}_{t} \Vert_{L^{2}(\Omega)}   \\
&+\int_{0}^{T^\ast}\Big(\Vert u^{k} \Vert^{2}_{W^{2,q_0}(\Omega)}\,+\, \Vert u^{k}_{t} \Vert^{2}_{H^{1}(\Omega)}\Big)\Big] 
\leq C\exp\Big( C \exp\big(C \mathcal{C}_0 \big) \Big).
\end{split}
\end{align}
Further details are skipped here and the reader is refer to \cite[Subsection 3.1]{choe_kim1} for full details. 
The estimates remain valid when taking the limit of the smooth approximations of $\mu$ and $\lambda$ to the actual less regular functions. The norm estimates imply that $T^\ast\leq T_k$ for all $k\geq 0$.  This finishes the proof of Lemma~\ref{lemma-estimates}.
\end{proof}

\subsection{Convergence of approximate solutions as $k\to\infty$}\label{sub:4.3}

\begin{lemma}\label{lemma:konvergence}
	Let  $\rho_0, u_0$ and $f$ as in Theorem~\ref{Main-Theorem}. Then for $\delta>0$ and  $\rho_0^\delta, u_0^\delta$ as in Subsection~\ref{subsec:approx_sol}	
	there exists unique solution to \eqref{depart} with regularity as in \eqref{reg-sol}.
\end{lemma}
\begin{proof}
We prove that the approximate solutions $(\rho^k, u^k)$, $k\geq 1$, constructed previously converges to a solution of the original problem \eqref{depart} in the strong sense as $k\to \infty$. %by adapting the startegy in \cite[Subsec. 3.2]{choe_kim1}. 
To this end, let us define
\begin{align*}
\overline{\rho}^{k+1} \,=\, \rho^{k+1}\,-\, \rho^{k} \quad \mathrm{and} \quad \overline{u}^{k+1} \,=\, u^{k+1}\,-\, u^{k} \quad \hbox{for } k\in \NN.
\end{align*}
Then using the  momentum equation in \eqref{iterating-scheme}, we have
\begin{align*}
\rho^{k+1} \overline{u}_t^{k+1} +\rho^{k+1} u^k \cdot \nabla \overline{u}^{k+1} +\divv \Ss u^{k+1} -\divv \Ss u^k +\nabla (p^{k+1} -p^k)
=\overline{\rho}^{k+1} (f - u_t^k -u^k \cdot \nabla u^k) -\rho^k \overline{u}^k \cdot \nabla u^k.
\end{align*}
Taking the scalar product of the above equation by $\overline{u}^{k+1}$ and integrating over $\Omega$, we get
\begin{align}\label{diff-conv-est}
\begin{split}
\Dt \intO &\rho^{k+1} |\overline{u}^{k+1}|^2 \, \dx + \intO (\divv \Ss u^{k+1} - \divv \Ss u^k) \cdot \overline{u}^{k+1} \, \dx  \\
& \leq C \intO \Big(|\overline{\rho}^{k+1}| |f - u_t^k - u^k \cdot \nabla u^k| \, |\overline{u}^{k+1}| + |\rho^k| |\overline{u}^k| |\nabla u^k| |\overline{u}^{k+1}| + |p^{k+1} - p^k|^2 \Big) \dx.
\end{split}
\end{align}
Let us estimate the second term in the left-hand-side of estimate \eqref{diff-conv-est}. Indeed, we denote by $J$
\begin{multline*}%\label{Lu(k+1)-L(uk)}
J:=\intO (\divv \Ss u^{k+1} - \divv \Ss u^k) \cdot \overline{u}^{k+1} \, \dx = 2\intO \Big(\mu(\vert \mathrm{D}(u^{k+1}) \vert^{2})\mathrm{D}(u^{k+1}) \,-\, \mu(\vert \mathrm{D}(u^{k}) \vert^{2})\mathrm{D}(u^{k})\Big): \nabla\overline{u}^{k+1}\,\mathrm{d}x\\
 +\intO \Big(\lambda(\mathrm{div}u^{k+1})\mathrm{div}u^{k+1} \,-\, \lambda(\mathrm{div}u^{k})\mathrm{div}u^{k}\Big)\mathrm{div}\overline{u}^{k+1}\,\mathrm{d}x.
 \end{multline*}
More generally, for $A$ and $B$ being two $d\times d-$matrix valued sufficiently smooth functions, we have  
\begin{align}\label{mu:A-B}
\begin{split}
\mu(\vert A\vert^{2})A \,-\, \mu(\vert B\vert^{2})B &= \int_{0}^{1}\frac{d}{d s}\Big(\mu(\vert s\,A \,+\, (1-s)\,B \vert^{2})(s\,A \,+\, (1-s)\,B)\Big)\,\mathrm{d}s\\
&=  \int_{0}^{1}\Big(\mu\big(\vert s\,A \,+\, (1-s)\,B \vert^{2}\big)  \,\, (A\,-\, B)\\
\hspace*{1.5cm} \, &+\,  2\mu'\big(\vert s\,A \,+\, (1-s)\,B \vert^{2}\big) \,\,\big( ( s\,A \,+\, (1-s)\,B) : (A-B)\big) \, (s\, A\,+(1-s)\, B)   \, \Big)\mathrm{d}s,
\end{split}
\end{align}
%&= \int_{0}^{1}\Big( \mu\big(\vert s\,A \,+\, (1-s)\,B \vert^{2}\big) \, + 2 \vert s\,A \,+\, (1-s)\,B \vert^{2}\mu'\big(\vert s\,A + (1-s)\,B \vert^{2}\big) \Big) \,\, (A\,-\, B)\, \mathrm{d}s
cf. also \cite[Equations (31)--(38)]{sarka}.
Similarly, for two sufficiently regular real valued functions $a$ and $b$, we have
\begin{align}\label{lambda:a-b}
\begin{split}
\lambda(a) a -\lambda(b) b &= \int_{0}^{1}\frac{d}{d s}\Big(\lambda(s\,a \,+\, (1-s)\,b)(s\,a \,+\, (1-s)\,b)\Big)\,\mathrm{d}x\\
&=\int_0^1\Big( \lambda(s\, a + (1-s)\, b) \, (a-b) \, + \,  \lambda^{\prime}(s\,a \,+\, (1-s)\,b)\,\,  (a-b)\,\, (s\,a \,+\, (1-s)\,b) \, \, \Big)\mathrm{d}s\\
&= \int_0^1 \Big(\lambda(s\, a + (1-s)\, b)  \, + \, (s\,a \,+\, (1-s)\,b) \, \,  \lambda^{\prime}(s\,a \,+\, (1-s)\,b) \Big) \,\, (a-b) \, \mathrm{d}s.
\end{split}
\end{align}
Again, depending on the sign of $\mu^\prime(\cdot)$, \eqref{eq:ellipticity} and by the virtue of \eqref{mu:A-B} we distinguish 
\begin{align*}
& \intO \Big(2\mu(\vert \mathrm{D}(u^{k+1}) \vert^{2})\mathrm{D}(u^{k+1}) \,-\, 2\mu(\vert \mathrm{D}(u^{k}) \vert^{2})\mathrm{D}(u^{k})\Big): \nabla\overline{u}^{k+1}\,\mathrm{d}x\\
& \geq \begin{cases}
	\displaystyle{\intO \int_0^1} 2\mu\big(\vert s\mathrm{D}(u^{k+1}) + (1-s)\mathrm{D}(u^{k})\vert^{2}\big) \; \vert\mathrm{D}(\overline{u}^{k+1})\vert^2   \,\ds \dx , & \hbox{if } \mu^\prime(\cdot) \geq 0, \vspace*{0.3cm}\\
	\displaystyle{\intO \int_0^1} 2\Big(\mu\big(\vert s\mathrm{D}(u^{k+1}) + (1-s)\mathrm{D}(u^{k})\vert^{2}\big) \vspace*{0.3cm} \\
	+ 2 \vert s\mathrm{D}(u^{k+1})+\, (1-s)\mathrm{D}(u^{k})\vert^{2} \, \mu'\big(\vert s\mathrm{D}(u^{k+1}) + (1-s)\mathrm{D}(u^{k})\vert^{2}\big) \Big)\,\, \vert\mathrm{D}(\overline{u}^{k+1})\vert^2   \, \,\ds\dx 	& \hbox{if } \mu'(\cdot) <0,
	\end{cases}\\
	& \geq 2\varepsilon_\mu\intO \vert\mathrm{D}(\overline{u}^{k+1})\vert^2\, \dx.
\end{align*}
%\footnote{\BT{Condition needed here $2\varepsilon_\mu + %3\varepsilon_\lambda^2>0$}}
Consequently, thanks to the above estimate and \eqref{lambda:a-b}, we infer that
\begin{equation}\label{J-estimate}
\begin{split}
J \geq 
%\varepsilon_\mu \intO |\D \overline{u}^{k+1} |^2 \,\dx + \varepsilon_{\lambda}^2  \intO \vert \divv \overline{u}^{k+1} \vert^2 \,\dx=
\varepsilon_\mu \intO |\nabla \overline{u}^{k+1} |^2 \,\dx + (\varepsilon_\mu+\varepsilon_{\lambda})  \intO \vert \divv \overline{u}^{k+1} \vert^2 \,\dx.
\end{split}
\end{equation}
Thus, using \eqref{J-estimate} and following the same lines as in \cite[Subsection 3.2]{choe_kim1}  we deduce from \eqref{diff-conv-est} and \eqref{J-estimate} the analogue of  \cite[Eq. (3.25)]{choe_kim1}, that is, 
\begin{align}\label{3.25}
\begin{split}
\Dt \Vert \sqrt{\rho^{k+1}} \:\overline{u}^{k+1} \Vert_{L^2}^2 \, +  \Vert \nabla \overline{u}^{k+1} \Vert_{L^2}^2 \leq B^{k}(t)\:
  \Vert \overline{\rho}^{k+1} \Vert_{L^{2}}^2  \, + \tilde{C}\: \Vert \sqrt{\rho^k} \: \overline{u}^k \Vert_{L^2}^2,
  \end{split}
\end{align}
where $B^k(t) = \tilde{C}\big( 1 + \|f\|_{L^3}^2 + \|\nabla u_t^k \|_{L^2}^2 \big)$. Note that $\int_0^{T_\ast} B^k (t) \, dt\leq \tilde{C}$ for all $k\geq 1$, thanks to the uniform bound \eqref{3.24}. Here we denote by $\tilde{C}$ a generic positive constant depending only on $\mathcal{C}_0$ and parameters of $C$.

Meanwhile, since the density $\rho$ satisfies the same equation as in \cite{choe_kim1}, %then 
establishing the analogue of estimate \cite[Equation (3.28)]{choe_kim1} in our case is straight forward. Indeed, we deduce that for any $\varepsilon>0$
\begin{align}\label{3.27-28}
\begin{split}
\Dt \|\overline{\rho}^{k+1}\|_{L^2}^2 \leq E^k_\varepsilon(t) \, \Vert \overline{\rho}^{k+1} \Vert_{L^2}^2 + \varepsilon \Vert \nabla \overline{u}^k \Vert_{L^2}^2,
\end{split}
\end{align}
where $E_\varepsilon^k(t)= C_\varepsilon \big(\|\rho^k(t)\|_{L^\infty} +\Vert \nabla \rho^k(t)\Vert_{L^3}  \big)^2 + C \Vert \nabla u^k (t)  \Vert_{L^\infty}$ and $C_\varepsilon>0$. By virtue of estimate \eqref{3.24}, we have $\int_0^t E_\varepsilon^k(s)\,ds \leq \tilde{C} + \tilde{C}_\varepsilon t$ for all $t\leq T^\ast$ and $k\geq 1$.

By combining \eqref{3.25} and \eqref{3.27-28} we deduce the analogue of  \cite[Eq. (3.29)]{choe_kim1}, which is the key point to prove that $(\rho^k, u^k)$ converges to $(\rho, u)$ in a strong sense, where $(\rho, u)$ is the solution to the original problem \eqref{depart} with initial data $(\rho_0^\delta, u_0^\delta)$. Moreover, by the lower semi-continuity of the norm,  the couple $(\rho, u)$ enjoys the following regularity estimate
\begin{align}\label{3.30}
\begin{split}
\esssup_{0<t<T_{\ast}}\Big(\Vert \rho &\Vert_{ W^{1,q_{0}}(\Omega)} \,+\, \Vert \rho_t \Vert_{L^{q_{0}}(\Omega)}\,+\, \Vert  u\Vert_{ H^{2}(\Omega)} \,+\, \Vert \sqrt{\rho}u_{t} \Vert_{L^{2}(\Omega)}\Big)\\
&+\, \int_{0}^{T_{\ast}} \Big(\Vert u(t) \Vert^{2}_{W^{2,q_{0}}(\Omega)} \,+\, \Vert u_{t}(t) \Vert^{2}_{H^{1}(\Omega)} \Big)\mathrm{d}t  \,\leq\, C\,\exp\big(C\exp(C C_{0})\big).
\end{split}
\end{align}
Further details are omitted here.
\end{proof}

\subsection{Conclusion of the existence proof: $\delta\to 0$}

We finish this section by proving the existence of a solution to system \eqref{depart} with initial data $(\rho_0, u_0,f)$ fulfilling the hypothesis mentioned in Theorem \ref{Main-Theorem}.
For each small $\delta>0$, let $\rho_0^\delta$ and  $u_0^\delta$ be  as  in Subsection~\ref{subsec:approx_sol}. 
Then according to the previous subsections, we know that there exists a time $T_{\ast}\in (0,T]$ and a unique strong solution $(\rho^\delta, u^\delta)$ in $[0,T_{\ast}]\times \Omega$ to the problem \eqref{depart} with the initial data $(\rho_0, u_0)$ replaced by $(\rho_0^\delta, u_0^\delta)$. In the following, we prove that the solution $(\rho^\delta, u^\delta$) equipped with the initial data $(\rho_0^\delta, u_0^\delta)$ converges to $(\rho, u)$ equipped with $(\rho_0, u_0)$. We start with the convergence of initial data given in the lemma below.
\begin{lemma}[Convergence of $u_0^\delta$ and $\rho_0^\delta$]\label{lemma:convergence_u0delta}
Let $\rho_0^\delta$ and $u_0^\delta$ for $\delta>0$ be as above. Then as $\delta \to 0$
\begin{align*}
	\rho_0^\delta &\to \rho_0 \quad \hbox{in}\quad W^{s,q}(\Omega)\quad\hbox{for}\quad s\in [0,1],\\
 u_0^\delta &\to u_0 \quad \hbox{in}\quad H^{s}(\Omega)\quad\hbox{for}\quad s\in [0,2), 
 \quad \hbox{and}\quad (u_0^\delta)_\delta\subset H^2(\Omega) \quad \hbox{bounded}. 
\end{align*}
\end{lemma}
\begin{proof}
%Next, we want to prove that $(\rho^\delta, u^\delta)$ converges to $(\rho, u)$ as $\delta$ tends to zero where $(\rho, u)$ is a solution to the problem \eqref{depart} with initial data $(\rho_0, u_0)$. Indeed, 
Starting with $\rho^\delta_0$, we know that since $\Omega$ is bounded that 
\begin{align}\label{eq:rho0delta}
	\norm{\rho_0^\delta-\rho_0}_{W^{s,q}(\Omega)} \; = C(\Omega, q)  \; \delta \to 0 \quad \hbox{as} \quad \delta\to 0 \quad\hbox{for}\quad s\in [0,1].
\end{align}
Turning now to $u_0^\delta$, we conclude from \eqref{eq:rho0delta} and the assumption in Theorem~\ref{Main-Theorem}
that
\begin{align}\label{eq:rho_delta}
	\begin{split}
	\norm{g(\rho_0^\delta)^{1/2} - g(\rho_0)^{1/2}}_{L^2(\Omega)} &\to 0 \quad  \hbox{as} \quad \delta\to 0, \quad \\
	\norm{\nabla p(\rho_0^\delta)-\nabla p(\rho_0)}_{L^q(\Omega)}= \norm{p'(\rho_0^\delta)\nabla \rho_0^\delta - p'(\rho_0)\nabla \rho_0}_{L^q(\Omega)}&\to 0 \quad \hbox{as} \quad \delta\to 0.
	\end{split}
\end{align}
In particular, using \eqref{initial-elliptic}    and
 Assumption~\ref{assumption-elliptic-reg} it follows that $u_0^\delta$ is  uniformly bounded in $H^2(\Omega)$ with respect to $\delta$.
Then by the Banach-Alaoglu theorem and compact embeddings there exist a limit function say $u^\prime$ and subsequences of $u_0^\delta$ (which with an abuse of notation we still name $u_0^\delta$) such that
\begin{align*}
  u_0^\delta  \rightharpoonup  u_0^\prime \quad \mbox{ in} \; H^2(\Omega), \qquad \mbox{and} \quad 
u_0^\delta \to u_0^\prime \quad  \mbox{ in} \; H^s(\Omega)\quad \hbox{for}\quad s\in [0,2).
\end{align*}
Next, to show that $u_0^\prime \equiv u_0$ one uses that
 $(\rho_0^\delta, u_0^\delta)$ and $(\rho_0,u_0)$ satisfies  \eqref{initial-elliptic} and \eqref{comp-cond}, respectively. Then by an estimate analogue to \eqref{J-estimate} below, one concludes that
\begin{multline*}
	\varepsilon_\mu \intO |\nabla (u_0^\delta-u_0) |^2 \,\dx + (\varepsilon_\mu+\varepsilon_{\lambda})  \intO \vert \divv (u_0^\delta-u_0) \vert^2 \,\dx \\ \leq \norm{u_0^\delta-u_0}_{L^2(\Omega)}\left(\norm{g(\rho_0^\delta)^{1/2} - g(\rho_0)^{1/2}}_{L^2(\Omega)} + \norm{\nabla p(\rho_0^\delta)-\nabla p(\rho_0)}_{L^2(\Omega)}\right) \to 0 \quad \hbox{as} \quad \delta\to 0,
\end{multline*}
because of \eqref{eq:rho_delta} and since $\norm{u_0^\delta-u_0}_{L^2(\Omega)}$ is bounded by construction.  So, by Poincar\'e's inequality $u_0^\delta \to u_0$ in $H^1(\Omega)$ for any subsequence, and because of the uniqueness of limits $u_0\equiv u_0^\prime$.
%This concludes that $\nabla u_0^\delta \to \nabla u_0$ in $L^2(\Omega)$, and consequently $u_0^\prime = u_0 + \rm{cte}$. Now, this implies that $u_0$ and $u_0^\prime$ satisfy the same equation \eqref{comp-cond}. By uniqueness of solutions, we deduce that $u_0 = u_0^\prime$ and that $(\rho_0^\delta, u_0^\delta)$ converges strongly in $(H^1)^2$ to $(\rho_0, u_0)$ solution to \eqref{comp-cond}. This finishes the proof.    
\end{proof}

We turn now to the convergence of $(\rho^\delta, u^\delta)$ to some $(\rho, u)$ as $\delta \to 0$. By the lower semi-continuity of the norm %and Lemma~\ref{lemma:convergence_u0delta}  
the corresponding solution $(\rho^\delta, u^\delta)$ satisfies the bound \eqref{3.30} 
%with $\mathcal{C}_0=\mathcal{C}_0(\rho_0^\delta, u_0^\delta)= \| g \|_{L^2}^2$ and the 
with constants %$C$ 
and $T_{\ast}$ independent of $\delta$. Therefore, by the Aubin-Lions lemma we  obtain a convergent subsequence $(\rho^{\delta_k}, u^{\delta_k})$
and a limit $(\rho,u)$
%-- that without the loss of generality we can still denote by $(\rho^k, u^k)$ -- 
such that %for $0<s,\sigma<1$ %-- which converges to a limit $(\rho, u)$ as $\delta_k\to 0$ in
\begin{align}\label{eq:convergence_rho_u}
	\begin{split}
	\rho^{\delta_k} &\to \rho \quad \hbox{in} \quad L^\infty(0,T^\ast;W^{s,q_0}(\Omega)), \quad s\in [0,1), \\
	u^{\delta_k} &\to u \quad \hbox{in} \quad L^2(0,T^\ast;W^{1+\sigma,q_0}(\Omega)), \quad \sigma\in [0,1),
	\end{split}
\end{align}
as $\delta_k\to 0$, where compact Sobolev embeddings are used.

Consider now the weak formulation of \eqref{depart} with regularized data $(\rho_0^\delta), u_0^\delta)$. Indeed, for a.e. $t\in (0,T_\ast)$ we have

\begin{multline*}
  \int_{\Omega} \rho^\delta(t, x) u^\delta(t, x) \psi(t,x)\,dx -  \int_0^{t} \int_{\Omega}\rho^{\delta}(s,x) \del_t \psi(s,x) \, dx\, ds
- \int_0^{t}\int_\Omega \rho^\delta(s,x)  u^\delta(s,x) \cdot \nabla \psi(s,x) \, dx \, ds \\ =\int_\Omega\rho^\delta_0(x)\psi(0,x) \, dx, \\
\int_\Omega \rho^\delta(t, x) u^\delta(t, x) \cdot \varphi (t, x) \, dx-\int_0^{t} \int_\Omega \rho^\delta(s,x) u^\delta(s,x)\cdot \del_t \varphi(s,x)\ dx\, ds
\\
- \int_0^t \int_\Omega(\rho^\delta(s,x) u^\delta\otimes u^\delta(s,x)) : \nabla \varphi(s,x)\, dx\,ds  
 +\int_0^{t} \int_{\Omega} \Ss u^\delta(s,x) : \nabla \varphi(s,x)  \, dx \, ds
 \\
 =  \int_\Omega \rho^\delta_0 (x) u^\delta_0(x) \cdot \varphi (0,x) \, dx +
\int_0^{t} \int_\Omega \rho^\delta(s,x)  f(s,x)\cdot \varphi(s,x) \, dx \, ds,
\end{multline*}
for all smooth functions $\psi$ %being  
%, satisfying $\psi(T) = 0$, 
and all smooth vector fields $\varphi$ with support in $[0,T_\ast]$, respectively.
%being  satisfying $\varphi(T)=0$. 
According to \cite{choe_kim1}, we know that the convergences above allow us to pass to the limit in all the terms of the weak formulation except the stress tensor which is the only term differing from the setting in \cite{choe_kim1}. 
However, for the stress tensor 
term we observe that
\begin{equation*}
    \| \Ss u^\delta \|_{L^{1}(0,T_*, L^{\infty}(\Omega)) \cap L^{2}(0,T_*, L^{6}(\Omega))} \leq C.
\end{equation*}
Moreover, by \eqref{eq:convergence_rho_u} we know that $\D(u^\delta)$ converges a.e. to $\D(u)$ in $(0, T^*) \times \Omega$, and thus by the continuity of $\mu(\cdot)$, we deduce that $\Ss u^\delta$ converges a.e. to $\Ss u$ in in $(0, T^*) \times \Omega$ too. Now, using Vitali's convergence therorem, we deduce that 
\[
\int_0^t \int_\Omega \Ss u^\delta : \nabla \varphi \, dx\, ds \to \int_0^t \int_\Omega \Ss u : \nabla \varphi  \, dx\, ds
\]
Hence, we can pass in the limit in all the terms of weak formulation. Thus $(\rho, u)$ is a weak solution to \eqref{depart} satisfying bounds \eqref{3.30} too, and hence a strong solution in the sense of Theorem~\ref{Main-Theorem}. 
%In particular, for the term $\Ss u^\delta$
%one has convergence as $\D(u^\delta)\to \D(u)$, $\mu(|\D(u^\delta)|^2) \to  \mu(|\D(u)|^2)$ and $\lambda(\divv u^\delta)\to \lambda(\divv u)$ in $L^2(0,T_\ast;L^\infty(\Omega))$.  
%Next, we prove that this solution is unique, and then this solution becomes a strong solution to \eqref{depart}. 
This finishes the proof of the existence of solutions in Theorem~\ref{Main-Theorem}.

\subsection{Uniqueness and continuous dependence on the data} 
Let $\rho^k$ and $u^k$ for $k\in\{1,2\}$ be two solutions to \eqref{depart} with \eqref{reg-sol} to  data $\rho_0^k,u_0^k$ and $f^k$ satisfying \eqref{regcond1}--\eqref{comp-cond}. Then set
\begin{align*}
\overline{\rho} &= \rho^{2}\,-\, \rho^{1}, \quad \overline{u} \,=\, u^{2}\,-\, u^{1} \quad \mathrm{and} \quad \overline{f} \,=\, f^{2}\,-\, f^{1}, \\
\rho^\ast &=  \rho^{2}\,+\, \rho^{1}, \quad u^\ast \,=\, u^{2}\,+\, u^{1} \quad \mathrm{and} \quad f^\ast \,=\, f^{2}\,+\, f^{1}.
\end{align*}
Using the  momentum and continuity equations in \eqref{depart}, we have 
\begin{align*}
\rho^{\ast} \overline{u}_t -2(\divv \Ss u^{2} -\divv \Ss u^1) 
+\tfrac{1}{2}\rho^\ast u^\ast\nabla \overline{u}
+2\nabla (p^{2} -p^1)
=(\overline{\rho}f^\ast+\rho^\ast \overline{f}) - \overline{\rho}u_t^\ast  
-\overline{\rho}(u_2\nabla u_2+u_1\nabla u_1)
-\tfrac{1}{2}\rho^\ast\overline{u}\nabla u^\ast.
\end{align*}
This can 
be estimated adapting the calculations in Subsection~\ref{sub:4.3} which implies uniqueness of solutions and continuous dependence on the data.

\subsection{Blow-up criterion}
 We finish the proof of Theorem \ref{Main-Theorem} by showing that the solution $(\rho, u)$ of system \eqref{depart} constructed previously blows up in finite time if \eqref{blowup} holds. The proof is again very similar to the one in \cite{choe_kim1}. Indeed, suppose that $T^\ast<T$, and let us introduce the functions %$\Phi(t)$ and $I(t)$
 \begin{align*}
 \Phi(t) &= 1  + \| \rho(t) \|_{W^{1,q_0}} + \| u(t) \|_{H^1}\quad \hbox{and} \\
  I(t) &= 1 + \| \rho(t) \|_{W^{1,q_0}} + \| \rho_t(t) \|_{L^2} + \| u(t) \|_{H^2} + \| \sqrt{\rho} u_t \|_{L^2} + \int_0^t \big( \| u(s) \|_{W^{2, q_0}}^2 + \| u_t(s) \|_{H^1}^2\big) \, \ds,
 \end{align*}
 for $0<t<T^\ast$. Let $\tau$ be a fixed time in $(0, T^\ast)$. Then $(\rho, u)$ is a strong solution to \eqref{depart} in $[\tau, T^\ast) \times \Omega$, which satisfies the regularity \eqref{reg-sol}. Following the same arguments as in Subsection ~\ref{sub4.1}, we can prove that for any $t\in (\tau, T^\ast)$ the following estimates hold
 \begin{align}
 \| \nabla u(t) \|_{H^1} &\leq C \big( 1+ \| \sqrt{\rho} u_t(t) \|_{L^2} \big) M(\Phi(t)), \label{4.5}\\
  \| \sqrt{\rho} u_t(t)\|_{L^2}^2 + \int_\tau^t \|\nabla u(s) \|_{L^2}^2 \, \ds &\leq C +C \, \| \sqrt{\rho} u_t (\tau) \|_{L^2}^2 + C \int_\tau^t \big( 1+ \| \sqrt{\rho} u_t \|_{L^2}^2\big) M(\Phi) \, ds, \label{4.6}\\
  \| \rho(t) \|_{W^{1,q_0}} &\leq C \exp\Big( C \int_0^t \|\nabla u(s) \|_{W^{1,q_0}} \, \ds \Big),\label{4.7} \quad \hbox {and} \\
%\intertext{and}
  \| \nabla u(t) \|_{W^{1, q_0}} &\leq C \Big( \big( 1+ \|\sqrt{\rho} u_t(t) \|_{L^2}^2 \big) M(\Phi(t)) + \|f(t) \|_{L^{q_0}}^2 + \|\nabla u_t(t) \|_{L^2}^2 \Big)  \label{4.8}
 \end{align}
 for an increasing continuous function $M \colon [0, \infty) \rightarrow [0, \infty)$.
 
 By the virtue of Gr\"onwall's inequality, we deduce from \eqref{4.6} that 
 \begin{equation}\label{4.9}
 \| \sqrt{\rho} u_t (t) \|_{L^2}^2 + \int_0^t \|\nabla u_t \|_{L^2}^2 \, \ds \leq C I(\tau) \exp \Big( C T^* \: \underset{0\leq s\leq t}{\sup} \: M(\Phi(s) \Big).
 \end{equation}
 By combining \eqref{4.5}, \eqref{4.7}-\eqref{4.9}, and using the continuity equation, we deduce that for any $t\in (\tau, T^*)$, 
 \begin{equation}\label{4.10}
 I(t) \leq C I(\tau) \Big( \underset{0\leq s \leq t}{\sup} \: M(\Phi(s))\Big) \exp\Big( C T^* \underset{0\leq s \leq t}{\sup} \: M(\Phi(s))\Big).
 \end{equation}
Hence, the blow-up criterion \eqref{blowup} follows immediately  from \eqref{4.10} because the maximality of $T^*$ implies that $J(t) \rightarrow \infty$ as $t \rightarrow T^\ast$. The proof of Theorem \ref{Main-Theorem} is now complete.

\subsection*{Acknowledgements} We would like to thank Andrea Cianchi and Michael R\r{u}\v{z}i\v{c}ka for their valuable comments on non-linear elliptic estimates. Also, we appreciate the very helpful comments by the referee. 

\subsection*{Declaration: Conflicts of interests/Competing interests}
None.
\subsection*{Declaration: Data availability}
Not applicable.

\bibliographystyle{abbrv}
\bibliography{Main}
\end{document}